\documentclass[12pt]{amsart}
\usepackage{amssymb}
\usepackage{amsmath}
\usepackage{amsthm}
\usepackage{changepage}
\usepackage{caption}
\usepackage{array}
\usepackage{enumitem}
\usepackage{graphicx}
\usepackage{xcolor}
\newtheorem{theorem}{Theorem}[section]

\newtheorem{corollary}[theorem]{Corollary}

\newtheorem{prop}[theorem]{Proposition}

\usepackage{subfig}
\usepackage{hyperref}
 
\theoremstyle{definition}

\newtheorem{example}[theorem]{Example}

\theoremstyle{remark}
\newtheorem{remark}[theorem]{Remark}

\numberwithin{equation}{section}

\def\Ker{{\text{Ker}}}

\def\deg{{\text{deg}}}

\begin{document}

\title{Mordell--Weil groups of quasi-elliptic surfaces in characteristic 3}


\author{Toshiyuki Katsura\\
ORCID: 0000-0003-4328-8887}
\address{Graduate School of Mathematical Sciences, The University of Tokyo, Tokyo,
153-8914, Japan}
\email{tkatsura@g.ecc.u-tokyo.ac.jp}

\thanks{This research is partially supported by JSPS KAKENHI Grant-in-Aid 
for Scientific Research (C) No.26K06737.}


\begin{abstract}
In characteristic $p =3$, we establish a concrete relationship between the torsion rank 
of the elementary abelian group of global sections of quasi-elliptic surfaces 
over the projective line ${\mathbf P}^1$ and the $p$-rank of certain hyperelliptic curves, 
thereby explaining a mysterious relationship
between these two invariants previously observed by I. Dolgachev. As an application, we determine
the torsion ranks of the Mordell--Weil groups of a class of quasi-elliptic surfaces.
\end{abstract}

\maketitle

\markboth{Mordell--Weil groups of quasi-elliptic surfaces}{T. Katsura}

\section{Introduction}
Unipotent groups over imperfect fields of positive characteristic are 
interesting objects
of study and are closely related to quasi-elliptic fibrations on smooth projective surfaces (\cite{BM} 
and \cite{CDL}). In \cite{D}, I. Dolgachev studied integral models of unipotent 
commutative algebraic groups over the field $K$ of rational functions 
of a Dedekind scheme $S$ over an algebraically closed field $k$
of positive characteristic and he discovered a mysterious relationship between the torsion rank of 
the elementary abelian group of global sections of unipotent algebraic goups and 
the $p$-rank of certain hyperelliptic curves in characteristic $p > 2$. In this paper, 
we take up this problem and establish a concrete relationship between the torsion rank of 
the elementary 
abelian group of global sections of quasi-elliptic surfaces over the projective line ${\mathbf P}^1$
and the $p$-rank of certain hyperelliptic curves in characteristic $p =3$.

In Section 2, we give a brief survey of quasi-elliptic surfaces 
over the projective line ${\mathbf P}^1$ in characteristic 3. We clarify
the normal forms of such surfaces and the equations in \cite{D}.
In Section 3, we clarify the relationship between the Frobenius map and the Cartier operator
for singular hyperelliptic curves in positive charateristic.
In Section 4, we prove that, in charactersitic 3, the Mordell--Weil group of of a quasi-elliptic surface 
with a section over ${\mathbf P}^1$ is isomorphic to a certain subgroup 
of the semisimple part of the first cohomology group ${\mathrm H}^1(C, {\mathcal O}_{C})$
of the structure sheaf of a certain
hyperelliptic curve $C$. In particular, if the quasi-elliptic surface is rational,
then the Mordell--Weil group is isomorphic to 
$\ker (F - {\rm id}) \subset {\mathrm H}^1(C, {\mathcal O}_{C})$. Here, $F$ is the Frobenius map
on ${\mathrm H}^1(C, {\mathcal O}_{C})$.
Finally, in Section 5, we apply our method to determine 
the torsion ranks of Mordell--Weil groups 
of quasi-elliptic K3 surfaces with 10 singular fibers of type ${\rm IV}$. 
In \cite[p. 25]{I}, Ito considered this problem and stated that, in this case,
only some examples 
would be given, since more complicated computations are required to determine
the torsions rank and the sections.
We apply our method to this case and provide a more complete picture 
of the possible torsion ranks and sections.

The author is deeply grateful to Professor I. Dolgachev for suggesting this topic.

\section{Preliminaries}
Let $k$ be an algebraically closed field of charactersitic 3. 
Let $f : S \longrightarrow {\mathbf P}^1$ be a relatively minimal quasi-elliptic 
surface with a section. Then $S$ is birationally equivalent to a surface define by
$$
      y^2 = x^3 + \varphi (t)
$$
for a suitable polynomial $\varphi (t) \in k[t]$ (cf. Miyanishi \cite{M}, Miyanishi-Ito \cite{MI}). 
Here, $t$ is a local parameter on the base curve ${\mathbf P}^1$.

Now make the change of coordinates 
$$
u = \frac{x}{y},\qquad v = \frac{1}{y}.
$$ 
Then, putting $a_{6n}(t) = \varphi (t)$,
we obtain 
\begin{equation}\label{equation}
 u^3 = v - a_{6n}(t) v^3.
\end{equation}
This is the equation appearing in \cite[Example 6.3]{D}. 
(The difference in sign is irrelevant.)
Therefore, any quasi-elliptic surface over ${\mathbf P}^1$  is birationally equivalent to a surface
given by the equation (\ref{equation}).

The group of sections of the quasi-elliptic surface $f : S \longrightarrow {\mathbf P}^1$
is called a Mordell--Weil group, which is denoted by ${\rm MW}(S)$.
Since $p = 3$ and $f : S \longrightarrow {\mathbf P}^1$ is quasi-elliptic, the order
of any non-zero element of ${\rm MW}(S)$ is $3$, and there exists a non-negative integer $r$
such that ${\rm MW}(S)\cong ({\mathbb Z}/3{\mathbb Z})^r$.
$r$ is called a torsion rank of ${\rm MW}(S)$ (cf. Ito \cite{I}).

\section{Frobenius map and Cartier operator}
We consider a hyperelliptic curve defined by
\begin{equation}\label{1}
    y^2 = a_{6n}(t)
\end{equation}
in the affine plane ${\mathbb A}^2$ where $a_{6n}(t)\in k[t]$ is a polynomial
satisfying
$6n -5 \leq \deg~ a_{6n}(t)\leq 6n -1$ ($n \in {\mathbb Z}$, $n \geq 1$). 
We set $\ell = \deg~ a_{6n}$.
We denote this affine curve by $U_0$.

We also consider the affine curve 
\begin{equation}\label{2}
    Y^2 = \tilde{a}_{6n}(T)
\end{equation}
in the affine plane ${\mathbb A}^2$, obtained by the change of coordinates
$$
       Y = \frac{y}{t^{3n}}, T = \frac{1}{t}.
$$
Then
$$
        \tilde{a}_{6n}(T) = a_{6n}(t)/t^{6n}.
$$
We denote this affine curve by $U_1$, and denote
by $C$ the curve defined by (\ref{1}) and (\ref{2}).
Thus, $C$ is a (not necessarily non-singular) hyperelliptic curve 
equipped with a morphism $\pi : C \longrightarrow {\mathbf P}^1$.
We denote by $P_0$ the point on ${\mathbf P}^1$ defined by $t=0$.
We have $\dim {\rm H}^1(C, {\mathcal O}_C)= 3n-1$.

In this section, we summarize some facts about the Frobenius map and the Cartier operator
on $C$ that will be used later.
If $C$ is non-singular, these facts are well known.
We set $g = 3n-1$. 
Then, the arithmetic genus of the curve $C$ is $g$.
The collection ${\mathcal U} =\{U_0, U_1\}$ gives an affine open covering of $C$.
Using the \v{C}ech cohomology with respect to ${\mathcal U}$, a basis of
${\rm H}^1(C, {\mathcal O}_C)$ is given by
$$
  \frac{y}{t^i} \quad {\rm on}~ U_0\cap U_1 ~(1 \leq i \leq g).
$$
We denote by $\Omega$ the vector space over $k$ generated by the rational $1$-forms
$$
                \frac{t^jdt}{y}\quad (0 \leq j \leq g -1).
$$
If the genus of the nonsingular model $\tilde{C}$ of $C$ is equal to
$g$, then $\Omega$ is isomorphic to the space 
${\rm H}^0 (\tilde{C}, \Omega_{\tilde{C}})$ of regular $1$-forms.

To construct the pairing between ${\rm H}^1(C, {\mathcal O}_C)$ and $\Omega$,
we consider an exact sequence
$$
   0\longrightarrow {\mathcal O}_C \longrightarrow {\mathcal K}_C \longrightarrow 
  {\mathcal K}_C/{\mathcal O}_C \longrightarrow 0.
$$
Here ${\mathcal K}_C$ is the constant sheaf associated with the rational function field $k(C)$.
We have the long exact sequence
$$
   \longrightarrow k(C) \longrightarrow {\rm H}^0(C, {\mathcal K}_C/{\mathcal O}_C) \stackrel{\delta}{\longrightarrow} 
   {\rm H}^1(C, {\mathcal O}_C) \longrightarrow {\rm H}^1(C, {\mathcal K}_C).
$$
Since ${\mathcal K}_C$ is a constant sheaf, we have ${\rm H}^1(C, {\mathcal K}_C) = 0$,
and hence $\delta$ is surjective. Since $\frac{y}{t^i}$ is a rational function on $U_0$,
the \v{C}ech cocycle $\beta = (\frac{y}{t^i}, 0)$ 
($1 \leq i \leq g$) with respect to 
the affine open covering ${\mathcal U}$ is mapped to $\frac{y}{t^i}$ by $\delta$.
Using these representatives, we define a pairing
$$
\begin{array}{ccc}
    {\rm H}^1(C, {\mathcal O}_C) \times \Omega & \longrightarrow & k \\
     (\delta ((\frac{y}{t^i}, 0)) , \frac{t^jdt}{y}) & \longmapsto & 
          \operatorname{Res}_{P_0}t^{j-i}dt.
\end{array}
$$
Here, $\operatorname{Res}_{P_0}$ denotes the residue at the point $P_0$. Since the Gram matrix
$(\langle \delta ((\frac{y}{t^i}, 0)), \frac{t^jdt}{y})\rangle)$ is the identity matrix,
this pairing is nondegenerate.
A direct calculation gives the following proposition.
\begin{prop}\label{adjoint}
The Cartier operator ${\mathcal C}$ on $\Omega$ is adjoint to the  Frobenius map $F$
on ${\rm H}^1(C, {\mathcal O}_C)$ with respect to the above pairing. That is,
$$
     \langle F(\alpha), \omega \rangle = \langle \alpha, {\mathcal C}(\omega) \rangle^3
$$
for $\alpha \in {\rm H}^1(C, {\mathcal O}_C)$ and $\omega \in \Omega$.
\end{prop}

Using this proposition, we obtain $\Ker (F - {\rm id}) \cong \Ker ({\mathcal C} - {\rm id})$.
Therefore, we use $\Omega$ in place of ${\rm H}^1(C, {\mathcal O}_C)$ 
in order to study Dolgachev's mysterious relationship.

\section{Mysterious relationship}         
In this section, we use the notaion introduced in the previous section.
Let $f : S \longrightarrow {\mathbf P}^1$ be a relatively minimal quasi-elliptic 
surface with a section which is birationally equivalent to the affine surface
defined by the equation (\ref{equation}). By a suitable change of coordinates
we may assume
\begin{equation}\label{standard}
    a_{6n}(t) = b_1(t)^3 t + b_2(t)^3 t^2
\end{equation}
with some polynomials $b_i(t) \in k[t]$ ($i= 1, 2$).
Moreover, we can assume, by a suitable change of coordinate,
the following two conditions (cf. Miyanishi \cite[Theorem 1]{M}, 
Ito \cite[Theorem 1.4]{I}).

\noindent
Assumption $(*)$

(1) For every root $\alpha$ of $a'_{6n}(t) =0$, $v_{\alpha}(a_{6n}(t)-a_{6n}(\alpha))\leq 5$,
where $v_{\alpha}$ is the $(t -\alpha)$-adic valuation of $k[t]$ with $v_{\alpha}(t -\alpha)=1$.

(2) If, moreover, $a_{6n}(t)-a_{6n}(\alpha) = a(t-\alpha)^3 +$ (terms of higher degree in $t - \alpha$)
for some root $\alpha$ of $a'_{6n}(t) =0$ and $a \in k^*$, then 
$v_{\alpha}(a_{6n}(t)-a_{6n}(\alpha) - a(t-\alpha)^3)\leq 5$.

\noindent
Here, we write the derivative $\frac{d}{dt}a_{6n}(t)$ as $a'_{6n}(t)$.

We examine the Cartier operator ${\mathcal C}$ on the curve $C$ defined by (\ref{1}) and (\ref{2}).
We can write
$$
        t^ia_{6n}(t) = h_0(t)^3 + h_1(t)^3t + h_2(t)^3t^2
$$
with some polynomials $h_j(t) \in k[t]$ ($j= 0, 1, 2$).
Using this expression, the Cartier operator ${\mathcal C}$ is given by
\begin{equation}\label{cartier}
   {\mathcal C}(\frac{t^idt}{y}) = {\mathcal C}(\frac{t^iy^2dt}{y^3})
   ={\mathcal C}(\frac{t^ia_{6n}(t)dt}{y^3}) = \frac{h_2(t)dt}{y}.
\end{equation}
Let $U$ be the vector space of polynomials in $k[x]$ of degree less than or equal to $g -1$.
For $f(t) \in U$, we can write
$$
   f(t)a_{6n}(t) =f_0(t)^3 + f_1(t)^3t + f_2(t)^3 t^2
$$
with some polynomials $f_i(t) \in k[t]$ ($i= 0, 1, 2$).
Then, considering (\ref{cartier}), the Cartier operator is given by the $p^{-1}$-linear
homorphism
$$
\begin{array}{rccc}
  \tilde{{\mathcal C}} :& U & \longrightarrow & U\\
         & f & \longmapsto & f_2(t)
\end{array}
$$
and we have
$$
   \ker ({\mathcal C} - {\rm id}_{\Omega}) \cong \ker (\tilde{{\mathcal C}} - {\rm id}_U).
$$
We set 
$$     
H = \ker (\tilde{{\mathcal C}} - {\rm id}_U).
$$
$H$ is a vector space over ${\mathbb F}_3$, and $\dim_{{\mathbb F}_3} H$ equals
the stable rank of the Frobenius map on ${\rm H}^1(C, {\mathcal O}_C)$
(Dolgachev \cite[Lemma 6.10]{D}).
We study the relationship between $H$ and ${\rm MW}(S)$.

For $p(t) \in H$, we write
$$
      p(t) = p_0(t)^3 + p_1(t)^3t + p_2(t)^3 t^2
$$
with some polynomials $p_i(t) \in k[t]$ ($i= 0, 1, 2$).
Since
$$
\begin{array}{l}
p(t)a_{6n}(t) \\
= (p_1(t)b_2(t)t + p_2(t)b_1(t)t)^3 + (p_0(t)b_1(t) + p_2(t)b_2(t)t)^3t \\
\quad +(p_0(t)b_2(t) + p_1(t)b_1(t))^3t^2,
\end{array}
$$     
we have 
\begin{equation}\label{l}
\tilde{{\mathcal C}}(p(t))= p_0(t)b_2(t) + p_1(t)b_1(t) 
 = p_0(t)^3 + p_1(t)^3t + p_2(t)^3 t^2 =p(t)
\end{equation}
by $p(t) \in H$.  
We set
$$
   H_{0} = \{ p(t) \in H\mid p_2(t) \equiv 0\}.
$$
Then $H_0$ is a subspace of $H$ over ${\mathbb F}_3$. 

Take $p(t)= p_0(t)^3 + p_1(t)^3t \in H_0$. Then, by (\ref{l}), $p(t) = p_0(t)b_2(t) + p_1(t)b_1(t)$.
Set $v = \frac{p(t)}{a'_{6n}(t)}$.
Then, 
$$
v - a_{6n}(t)v^3 = \frac{p(t)a'_{6n}(t)^2 -a_{6n}(t)p(t)^3}{a'_{6n}(t)^3}.
$$
We have 
$$
\begin{array}{l}
p(t)a'_{6n}(t)^2 -a_{6n}(t)p(t)^3\\
=(p_0(t)^3 + p_1(t)^3t)(b_1(t)^3 -b_2(t)^3t)^2 
- (b_1(t)^3t + b_2(t)^3t^2)(p_0(t)b_2(t) + p_1(t)b_1(t))^3\\
= (p_0(t)b_1(t)^2 + p_1(t)b_2(t)^2t)^3.
\end{array}
$$
Therefore, 
$$
(u, v) = (\frac{p_0(t)b_1(t)^2 + p_1(t)b_2(t)^2t}{a'_{6n}(t)}, \frac{p(t)}{a'_{6n}(t)})
$$
gives a rational point on the curve defined by (\ref{equation}) over $k(t)$.  Hence, it
gives an element of ${\rm MW}(S)$.
We consider a homomorphism
$$
\begin{array}{rccc}
\varphi :& H_0  & \longrightarrow & {\rm MW}(S) \\
       & p(t)   & \longmapsto & 
       (u, v) = (\frac{p_0(t)b_1(t)^2 + p_1(t)b_2(t)^2t}{a'_{6n}(t)}, \frac{p(t)}{a'_{6n}(t)}).
\end{array}    
$$

\begin{theorem}\label{main} 
$\varphi$ is an isomorphism.
\end{theorem}
\begin{proof}
$\varphi$ is clearly an injective homomorphism.
It remains to prove that $\varphi$ is surjective.

Take $(u, v) \in {\rm MW}(S)$. Then there exist rational functions $\theta(t), \eta (t) \in k(t)$
such that $u = \theta (t)$ and $v = \eta (t)$.  If $\eta (t) \equiv 0$, then $\theta (t)$ is also 0
and we can take $p(t) = 0$. Therefore, we assume that $\eta (t) \not\equiv 0$.
If We can write
$$
     \theta (t) = \frac{r(t)}{q(t)}, \quad \eta (t) = \frac{g(t)}{f(t)}
$$
with $q(t), r(t), f(t), g(t) \in k[t]$. We may assume that $q(t)$ and $r(t)$ (resp. $f(t)$ and
$g(t)$) are relatively prime. We have
\begin{equation}\label{solution}
   \theta (t)^3 = \frac{g(t)}{f(t)} -a_{6n}(t)\left(\frac{g(t)}{f(t)}\right)^3.
\end{equation}
Differentiating both sides with resopect to $t$, we have an equation
$$
(g'(t)f(t) -g(t)f'(t))f(t) = a'_{6n}(t)g(t)^3.
$$
Since $f(t)$ and $g(t)$ are relatively prime, $f(t)$ divides $a'_{6n}(t)$.
Therefore, there exists a polynomial $p(t) \in k[t]$ such that
$$
    \eta (t) =\frac{p(t)}{a'_{6n}(t)}.
$$
Putting this expression into (\ref{solution}), we have an equation
$$
    a'_{6n}(t)^3r(t)^3 =(p(t)a'_{6n}(t)^2 - a_{6n}(t)p(t)^3)q(t)^3.
$$
Since $q(t)$ and $r(t)$ are relatively prime, it follows that 
$q(t)$ divides $a'_{6n}(t)$.
Hence, there exists a polynomial $h(t) \in k[t]$ such that 
$$
   \theta (t) = \frac{h(t)}{a'_{6n}(t)}.
$$
(It includes the case $h(t)= 0$.) 
Substituting these results into (\ref{equation}), we obtain an equation
\begin{equation}\label{h(t)}
   h(t)^3 = a'_{6n}(t)^2p(t) - a_{6n}(t)p(t)^3.
\end{equation}
Let $p(t) = p_0(t)^3 + p_1(t)^3t + p_2(t)^3 t^2$ where
$p_0(t), p_1(t), p_2(t) \in k[t]$.
Then, by (\ref{h(t)}) we have
$$
\begin{array}{rl}
h(t)^3 =& (b_1(t)^2p_0(t) + b_1(t)b_2(t)p_2(t)t + b_2(t)^2p_1(t)t)^3\\
     & + (b_1(t)^2p_1(t) + b_2(t)^2p_2(t)t + b_1(t)b_2(t)p_0(t)- b_1(t)p(t))^3t\\
     & +(b_1(t)^2p_2(t) + b_2(t)^2p_0(t) + b_1(t)b_2(t)p_1(t) -b_2(t)p(t))^3t^2.
\end{array}
$$
Since $\operatorname{char} = 3$, we have $k[t] =k[t^3] \oplus k[t^3]t \oplus k[t^3]t^2$.
In particular, every polynomial in $k[t]$ can be written uniquely in the form
$$
f_0(t)^3 + f_1(t)^3t + f_2(t)^3t^2\qquad (f_0(t), f_1(t), f_2(t)\in k[t]).
$$
Therefore, comparing the components corresponding to the powers of $t$ modulo $3$,
we obtain
\begin{equation}\label{3equations}
\begin{array}{l}
     (1) \quad h(t) =b_1(t)^2p_0(t) + b_1(t)b_2(t)p_2(t)t + b_2(t)^2p_1(t)t\\
     (2) \quad b_1(t)^2p_1(t) + b_2(t)^2p_2(t)t + b_1(t)b_2(t)p_0(t)- b_1(t)p(t)=0\\
     (3) \quad b_1(t)^2p_2(t) + b_2(t)^2p_0(t) + b_1(t)b_2(t)p_1(t) -b_2(t)p(t) = 0
\end{array}
\end{equation}
Taking $(2) \times b_2(t) - (3) \times b_1(t)$, we obtain
$$
      (b_2(t)^3 t - b_1(t)^3)p_2(t) = 0.
$$
Since $a_{6n}(t)$ is not identically zero, $b_2(t)^3 t - b_1(t)^3$ is not identically zero.
Hence, $p_2(t) \equiv 0$.    
Therefore, by $(1)$ we obtain
$$
       h(t) = b_1(t)^2p_0(t) + b_2(t)^2p_1(t)t.
$$ 
Therefore, we conclude that
$$
(u, v) = (\frac{p_0(t)b_1(t)^2 + p_1(t)b_2(t)^2t}{a'_{6n}(t)}, \frac{p_0(t)^3 + p_1(t)^3t}{a'_{6n}(t)}).
$$
Since at least one of $b_1(t) \not\equiv 0$ and $b_2(t) \not\equiv 0$ holds, (2) or (3) gives
\begin{equation}\label{p(t)}
 p(t) = p_0(t)^3 + p_1(t)^3t = b_1(t)p_1(t) + b_2(t)p_0(t).
\end{equation}
It remains to show $p(t) = p_0(t)^3 + p_1(t)^3t \in H_0$.   
By \eqref{p(t)} and the definition of $\tilde{{\mathcal C}}$, we have
$$
\tilde{\mathcal C} (p(t))= \tilde{\mathcal C} (p_0(t)^3 + p_1(t)^3t) = b_1(t)p_1(t) + b_2(t)p_0(t) = p(t).
$$
Thus, it remains only to show that $\deg~ p(t) \leq g - 1$.
Set $\deg~ p(t) = m$. Since $\deg~p_0(t)^3 \neq \deg~(p_1(t)^3t)$, we have $3 \deg ~p_0 \leq m$
and $3 \deg ~p_1 + 1 \leq m$. Similarily, since  $a_{6n}(t) = b_1(t)^3t + b_2(t)^3t^2$, 
we have $3 \deg~ b_1(t) + 1 \leq \ell$ and $3\deg~b_2(t) + 2 \leq \ell$.
By \eqref{p(t)}, we have
$$
\begin{array}{rl}
  m = \deg~p(t) & \leq \max \{\deg~ b_1(t) + \deg~p_1(t), \deg~ b_2(t) + \deg~p_0(t)\} \\
       & \leq \max \{\frac{\ell -1}{3} + \frac{m -1}{3}, \frac{\ell -2}{3} + \frac{m}{3}\}\\
       & = \frac{\ell + m -2}{3}. 
\end{array}
$$
Therefore, we have $m \leq \frac{\ell}{2} -1 \leq 3n - \frac{3}{2}$.
Since $m$ is an integer, we conclude $m \leq g -1$ as desired.      
\end{proof}

\begin{corollary} 
If $f : S \longrightarrow {\mathbf P}^1$ is a rational quasi-elliptic surface,
then we have $H = H_0  \cong {\rm MW}(S)$.
\end{corollary}
\begin{proof}
Since $S$ is rational, Assumption $(*)$ implies that $\deg ~{a_{6n}(t)}\leq 5$ 
(cf. Miyanishi \cite[Theorem 1]{M}, Ito \cite[Theorem 1.4]{I}).
Hence, the arithmetic genus of the curve $C$ is less than or equal to $2$. It follows that,
for every $p(t) \in H$, we have $\deg~ p(t) \leq 1$. Consequently, $p_2(t) = 0$,
and thus $H = H_0$.
\end{proof}

This corollary recovers Dolgachev's result \cite[Example 6.13]{D} 
that Dolgachev's mysterious conjecture holds when the quasi-elliptic surface
$S$ is rational.

\section{Quasi-elliptic K3 surfaces with 10 singular fibers of type $IV$}
In this section we use the notation of the previous sections.
Let $f: S \longrightarrow {\mathbb P}^1$ be a quasi-elliptic K3 surface
with 10 singular fibers of type ${\rm IV}$. Then, by Ito \cite[Lemma 4.2 and Theorem 4.3]{I}
we may assume 
$$
   a_{6n}(t) = t^{10} + a_8^3t^8 + a_7^3t^7 + a_5^3t^5 + a_4^3t^4 + t^2\quad\quad 
   (a_4, a_5, a_7, a_8 \in k, n=2)
$$
and that $a'_{6n}(t) = 0$ has only simple roots. Here, we use $a_i^3$ in place of
$a_i$ for covenience in the calculation below; this does not cause any loss of generality.
In this case, the arithmetic genus of the curve $C$ is $3n-1=5$. For $p(t) \in H$,
we have $\deg~ p(t) \leq 4$, and hence $p(t)$ can be written as
$$
 p(t) = (c_0 + c_1t)^3 + (c_2 + c_3t)^3t + c_4^3t^2\qquad (c_i \in k, i=0, 1, \ldots, 4).
$$
Since 
$$
a_{6n}(t) = (t^3 + a_7t^2 + a_4t)^3t + (a_8t^2 + a_5t + 1)^3t^2,
$$
we have
$$
\tilde{\mathcal C}(p(t)) = 
(c_0 + c_1t)(a_8t^2 + a_5t + 1) +(c_2 + c_3t)(t^3 + a_7t^2 + a_4t).
$$
Since $p(t) \in H$, we have $\tilde{\mathcal C}(p(t)) = p(t)$. Comparing coefficients
on both sides, we obtain
\begin{equation}\label{5equations}
\left\{
\begin{array}{l}
    c_3 = c_3^3\\
    a_8c_1 + a_7c_3 + c_2 = c_1^3\\
    a_5c_1 + a_8c_0 + a_4c_3 + a_7c_2 = c_4^3\\
    c_1 + a_5c_0 + a_4c_2 = c_2^3\\
    c_0 = c_0^3.
\end{array}
\right.
\end{equation}
Therefore, $c_0, c_3 \in {\mathbb F}_3$. Writing $c_0 = \alpha$ and $c_3 = \beta$ 
for some $\alpha, \beta \in {\mathbb F}_3$, the above equations become
\begin{equation}\label{3}
\left\{
\begin{array}{l}
    (1) \qquad c_2 = c_1^3 - a_8c_1 - a_7\beta \\
    (2) \qquad c_1 = c_2^3 - a_4c_2 - a_5\alpha  \\
    (3) \qquad c_4^3 = a_5c_1 + a_7c_2 + a_8\alpha  + a_4\beta \\
\end{array}
\right.
\end{equation}
Substituting $(2)$ into $(1)$, we have the equation
\begin{equation}\label{c_2}
  c_2^9 -(a_4^3 + a_8)c_2^3 + (a_4a_8 -1)c_2 - a_5^3\alpha  + a_8a_5\alpha  -a_7\beta  = 0.
\end{equation}
If $a_4a_8 \neq 1$, then this equation has 9 distinct roots. 
If $a_4a_8 = 1$ and $a_4^3 + a_8 \neq 0$,
then it has 3 distinct roots. Finally, if $a_4a_8 = 1$ and $a_4^3 + a_8 = 0$, then
this equation has a unique root.
For each solution $c_2$, equations $(2)$ and $(3)$ uniquely determine
$c_1$ and $c_4$, respectively. Thus, we obtain the following result concerning
the stable rank $s = \dim_{{\mathbb F}_3}H$ of the Frobenius map
on ${\rm H}^1(C, {\mathcal O}_C)$.

\begin{theorem}\label{H}
$s = \dim_{{\mathbb F}_3} H \leq 4$. More precisely,
\begin{itemize}
\item[$({\rm i})$] If $a_4a_8 \neq 1$, $s = 4$.
\item[$({\rm ii})$] If $a_4a_8 = 1$ and $a_4^3 + a_8 \neq 0$, $s = 3$.
\item[$({\rm iii})$] If $a_4a_8 = 1$ and $a_4^3 + a_8 = 0$, $s = 2$.
\end{itemize}
\end{theorem}

\begin{corollary}
The Mordell--Weil torsion rank $r$ of the quasi-elliptic K3 surface $f : S \longrightarrow {\mathbf P}^1$
with 10 singular fibers of type ${\rm IV}$ is at most 4.
Moreover, the following hold:
\begin{itemize}
\item[$({\rm i})$] If $a_4a_8 = 1$, then $r \leq 3$.
\item[$({\rm ii})$] If $a_4a_8 = 1$ and $a_4^3 + a_8 = 0$, then $r \leq 2$.
\end{itemize}
\end{corollary}
\begin{proof}
This follows from Theorems \ref{main}, \ref{H}, togather with $H_0\subset H$.
\end{proof}

The bound $r \leq 4$ was also obtained by Ito \cite[Corollary 6.4]{I} by a different method.

The precise value $r$ is given by the following proposition.

\begin{theorem}\label{r}
The set of solutions $(c_0, c_1, c_2, c_3)$ of the system of equations 
$$
\left\{
\begin{array}{l}
    c_0 = c_0^3\\
    c_3 = c_3^3\\
    c_1 = c_2^3 - a_4c_2 - a_5c_0  \\
    c_2^9 -(a_4^3 + a_8)c_2^3 + (a_4a_8 -1)c_2 + (a_8a_5- a_5^3)c_0 - a_7c_3 = 0\\
    a_8c_0 + a_5c_1  + a_7c_2 + a_4c_3 =0\\
\end{array}
\right.
$$
form a finite abelian group isomorphic to $({\mathbb Z}/3{\mathbb Z})^r$, where $r$ is
the Mordell--Weil torsion rank of the quasi-elliptic K3 surface $f : S \longrightarrow {\mathbf P}^1$
with 10 singular fibers of type ${\rm IV}$.
\end{theorem}
\begin{proof} 
It is clear that the set of solutions forms a finite additive group. Since the characteristic is 3,
every element of this group has order 3, except for the zero element.
Hence, the group is isomorphic to $({\mathbb Z}/3{\mathbb Z})^r$ for some non-negative integer $r$.
By (\ref{5equations}) and Theorem \ref{main}, this $r$ is precisely 
the Mordell--Weil torsion rank of the given quasi-elliptic K3 surface.
\end{proof}

We now calculate some examples using this theorem. First, we consider the examples
in Ito \cite[Example 4.10]{I}. We denote by $s$ the stable rank of the Frobenius map
on ${\rm H}^1(C, {\mathcal O}_C)$, and
by $r$ the Mordell--Weil torsion rank of the quasi-elliptic K3 surface $f: S \longrightarrow {\mathbb P}^1$.

(1) Let $a_{12}(t) = t^{10} + a^3t^8 + t^7 + t^2$ where $a^3$ is a primitive fifth root of unity.

Since $\operatorname{char} k = 3$, $a$ is also a primitive fifth root of unity.
By Corollary \ref{H}, we have $s = 4$.
In this case the equations \eqref{3} become
$$
\left\{
\begin{array}{l}
    ({\rm i}) \qquad c_2 = c_1^3 - a c_1 - \beta\\
    ({\rm ii}) \qquad c_1 = c_2^3 \\
    ({\rm iii}) \qquad c_4^3 = c_2+ \alpha a\\
\end{array}
\right.
$$
If $p(t) \in H_0$, then $c_4 = 0$. Therefore, by $({\rm iii})$ we have $c_2 = -\alpha a$.
It follows from (ii) that  $c_1 = - \alpha a^3$, since $\alpha \in {\mathbb F}_3$.
Substituting these values of $c_1$ and $c_2$ into (i), and using $a^5 = 1$ we obtain
$\alpha a = \beta$. Since $\alpha, \beta \in {\mathbb F}_3$ and $a$ is 
a primitive fifth root of unity, this equation can hold only if $\alpha = \beta = 0$.
Therefore, $c_1 = c_2 = 0$. Hence $p(t) =0$. This shows that $H_0 = \{0\}$,
and consequently $r = 0$.

In a similar way, we obtain the following results for the remaining examples 
in \cite[Examples 4.10]{I},
which are consistent with the results obtained there.

(2)  Let $a_{12}(t) = t^{10} + t^5 + t^2$. \quad $H_0 =\langle t \rangle$, $r = 1, s = 4$.

(3)  Let $a_{12}(t) = t^{10} + t^7 + t^4 + t^2$. \quad $H_0 =\langle 1, t^4 -t \rangle$,  
$r = 2, s = 4$.

(4)  Let $a_{12}(t) = t^{10} + t^8 + t^2$.

Let $\gamma$ be a root of $x^8 - x^2 - 1 = 0$. 
Then $H_0 =\langle t^2, \gamma t^3, \gamma^3t^3 + 1)\rangle$ and $r = 3, s = 4$.

(5)  Let $a_{12}(t) = t^{10} + t^2$.

Let $\gamma$ be a primitive eighth root of unity. 
Then $H_0 =\langle 1, \gamma t, \gamma^3t^3, t^4 \rangle$ and $r = 4, s = 4$.

\begin{theorem}\label{r=4}
Let $f : S \longrightarrow {\mathbf P}^1$ be a quasi-elliptic K3 surface 
with 10 singular fibers of type ${\rm IV}$. If the Mordell--Weil torsion rank $r$ is maximal, 
namely $r =4$,
the surface is unique up to isomorphism and is given by $y^2 = x^3 - (t^{10} + t^2)$ .
\end{theorem}
\begin{proof} For $r = 4$, all roots of (\ref{c_2}) are distinct, and must satisfy $c_4 = 0$ in (\ref{3}). Substituting (1) in (\ref{3}) into (3) in (\ref{3}), 
we obtain
$$
  c_4 = a_5c_2^3 +(a_7 -a_5a_4)c_2 + (a_8 - a_5^2)\alpha + a_4\beta.
$$
The equation (\ref{c_2}) has 9 roots, whereas the polynomial in $c_2$ defined by
$c_4= 0$ has degree  3. Therefore, this polynomial must be identically zero. 
Thus, 
$$
  a_5 =0, a_7 -a_5a_4= 0, (a_8 - a_5^2)\alpha + a_4\beta = 0.
$$
Since $\alpha$ and $\beta$ are arbitrary elements of ${\mathbb F}_3$,
we further obtain
$$
   a_8 - a_5^2 =0, a_4 =0.
$$
Hence, $a_4 = a_5 = a_7 = a_8 =0$, as desired.
\end{proof}

Finally, we give some examples in which $s$ is smaller than $4$.
\begin{example} 
(1) Let $a_4 = a_5 = a_8 = 1$ and $a_7 = -1$. 
Nemely, $a_{12}(t) =t^{10} +  t^8 - t^7 + t^5 + t^4 + t^2$. 
Then, $a_4a_8 = 1$ and $a_4^3 + a_8 = - 1 \neq 0$. Therefore, by Theorem \ref{H}
we have $s = 3$. In this case, by (\ref{c_2}), we have 
$c_2^3 + c_2 + \beta = 0$. It follows that 
$$
c_4^4 = c_1 + \alpha + \beta -c_2 = c_2^3 + c_2 + \beta = 0.
$$
Therefore, we obtain $r = s = 3$.

(2) Let $\xi$ be a primitive eighth root of unity.
Let $a_4 = \xi$, $a_8 = \xi^7$ and $a_5 = a_7 = 0$. 
Nemely, $a_{12}(t) =t^{10} + \xi^7 t^8  + \xi t^4 + t^2$. 
Since $a_4a_8 = 1$ and $a_4^3 + a_8 = 0$,
Theorem \ref{H} gives $s = 2$.
By (\ref{c_2}), we obtain $c_2 = 0$. 
By the equations in Theorem \ref{r}, we have 
$$
c_0 = \alpha, c_3 = \beta, \xi^7\alpha + \xi \beta= 0 \quad \quad (\alpha, \beta \in {\mathbb F}_3).
$$
Therefore, we have $\beta = - \xi^6 \alpha$.
However, $\xi^6 \not\in {\mathbb F}_3$. Therefore, $\alpha = \beta = 0$
and hence $c_1 = 0$. Consequently, we obtaion $r = 0$.
\end{example}

\begin{remark}
By a similar calculation to the one in the above examples, it is possible to solve 
the system of equations in Theorem \ref{r} and 
hence to determine the number of solutions explicitly. This allows us to determine
the Mordell--Weil torsion rank $r$. Although the resulting calculation is somewhat tedious, 
it yields the following classification.

\medskip
\noindent
\textbf{(I) Case: $a_5 \neq 0$.}

\textbf{(I-1) Case: $a_7^3(a_7 - a_4a_5)+a_5^3(a_7a_8 -a_5)= 0$.}

Set
$$
\begin{array}{rl}
 A_{\alpha}& = (a_5^2-a_8)(a_5a_8^2 - a_7^3),\\
 A_{\beta} &= a_4a_7^3 +a_4a_5^3a_8 -a_4^3a_5 -a_5^4a_7.
\end{array}
$$

\begin{itemize}
\item[({\rm i})] If $A_{\alpha} \neq 0$, $A_{\beta} \neq 0$, 
$A_{\beta}/A_{\alpha} \not\in {\mathbb F}_3$ and $a_7 - a_4a_5 = 0$, then $r = 0$.
\item[({\rm ii})] If $A_{\alpha} \neq 0$, $A_{\beta} \neq 0$, 
$A_{\beta}/A_{\alpha} \not\in {\mathbb F}_3$ and $a_7 - a_4a_5 \neq 0$, then $r = 1$.
\item[({\rm iii})] If $A_{\alpha} = A_{\beta} = 0$ and $a_7 -a_4a_5= 0$, then $r=2$.
\item[({\rm iv})] If $A_{\alpha} = A_{\beta} = 0$ and $a_7 -a_4a_5 \neq 0$, then $r=3$.
\item[({\rm v})] If $(A_{\alpha}, A_{\beta})\in {\mathbf P}^1({\mathbb F}_3)$ and
$a_7 -a_4a_5= 0$, then $r = 1$.
\item[({\rm vi})] If $(A_{\alpha}, A_{\beta})\in {\mathbf P}^1({\mathbb F}_3)$ and
$a_7 -a_4a_5 \neq 0$, then $r = 2$.
\end{itemize}

\textbf{(I-2) Case: $a_7^3(a_7 - a_4a_5)+a_5^3(a_7a_8 -a_5) \neq 0$.}

Set
$$
\begin{array}{rl}
 B_{\alpha} & = a_5(a_8 - a_5^2)^3(a_7^3 -a_5a_8^2)^3 \\
    &\quad - (a_4a_5-a_7)(a_8 -a_5^2)(a_7^3 - a_5a_8^2)(-a_4a_7^3a_5 + a_7^4 + a_8a_7a_5^3 -a_5^4)^2\\
  & \quad + (a_5^2 - a_8)(-a_4a_7^3a_5 + a_7^4 + a_8a_7a_5^3 -a_5^4)^3,  \\
B_{\beta} &= a_5(a_4a_7^3 + a_4a_8a_5^3 - a_4^3a_5 - a_7a_5^4)^3 \\
  &\quad  -(a_4a_5 -a_7)(a_4a_7^3 + a_4a_8a_5^3 - a_4^3a_5 - a_7a_5^4)(-a_4a_7^3a_5 + a_7^4 + a_8a_7a_5^3 - a_5^4)^2\\
  & \quad -a_4(a_4a_7^3a_5 + a_7^4 + a_8a_7a_5^3 -a_5^4)^3.
\end{array}
$$

\begin{itemize}
\item[({\rm i})] If $B_{\alpha} = B_{\beta} = 0$, then $r=2$.
\item[({\rm ii})] If $B_{\alpha} \neq 0$, $B_{\beta} \neq 0$ and 
$B_{\beta}/B_{\alpha} \not\in {\mathbb F}_3$, then $r = 0$.
\item[({\rm iii})] Otherwise, $r= 1$.
\end{itemize}

\noindent 
\textbf{(II) Case: $a_5 = 0$.}

\textbf{(II-1) Case: $a_7 \neq 0$.}

Set
$$
\begin{array}{rl}
  C_{\alpha} &= a_8^9 -(a_4^3 + a_8)a_7^6a_8^3 +(a_4a_8 -1)a_7^8a_8,\\
  C_{\beta} &= a_4^9 -(a_4^3 + a_8)a_4^3a_7^6 + (a_4a_8 -1)a_4a_7^8 + a_7^{10}.
\end{array}
$$
Then:
\begin{itemize}
\item[({\rm i})] If $C_{\alpha} = C_{\beta} = 0$, then $r=2$.
\item[({\rm ii})] If $C_{\alpha} \neq 0$, $C_{\beta} \neq 0$ and 
$C_{\beta}/C_{\alpha} \not\in {\mathbb F}_3$, then $r = 0$.
\item[({\rm iii})] Otherwise, $r= 1$.
\end{itemize}

\textbf{(II-2) Case: $a_7 = 0$.}

\quad \textbf{(II-2-1) Case: $a_4\neq 0$, $a_8 \neq 0$ and $a_8/a_4 \not\in {\mathbb F}_3$.}
\begin{itemize}
\item[({\rm i})] If $a_4a_8 \neq 1$, then $r=2$.
\item[({\rm ii})] If $a_4a_8 = 1$ and $a_3^3 + a_8 \neq 0$, then $r = 1$.
\item[({\rm iii})] If $a_4a_8 = 1$ and $a_3^3 + a_8 = 0$,  $r= 0$.
\end{itemize}

\quad \textbf{(II-2-2) Case: either $a_4 = 0$, or $a_4 \neq 0$ and $a_8/a_4 \in {\mathbb F}_3$.}

\begin{itemize}
\item[({\rm i})] If $a_4 = 0$ and $a_8 = 0$, then $r=4$.
\item[({\rm ii})] If $a_4a_8 \neq 1$ and $(a_4, a_8) \neq (0, 0)$, then $r=3$.
\item[({\rm iii})] If $a_4a_8 = 1$ and $a_3^3 + a_8 \neq 0$, then $r = 2$.
\item[({\rm iv})] If $a_4a_8 = 1$ and $a_3^3 + a_8 = 0$, then $r= 1$.
\end{itemize}

\end{remark}

For K3 surfaces, the cases other than the one with ten singular fibers of 
type ${\rm IV}$ can also be treated by our method. However, 
they have already been analyzed in detail by Ito \cite{I}.

\section*{Statements and Declarations}

\textbf{Competing Interests.}
The author declares that there are no competing interests.

\end{document}